\documentclass[a4paper]{article}
\usepackage{graphicx}
\usepackage{amsmath,amssymb}
\usepackage{caption}
\usepackage{subfig}
\newcommand{\ba}{\begin{array}}
\newcommand{\ea}{\end{array}}

\newtheorem{theorem}{Theorem}[section]
\newtheorem{lemma}[theorem]{Lemma}

\def\be{\begin{equation}}
\def\ee{\end{equation}}
\def\bea{\begin{eqnarray}}
\def\eea{\end{eqnarray}}
\def\a{\alpha}

\def\b{\beta}
\def\g{\gamma}
\def\d{\delta}
\def\l{\lambda}
\def\om{\omega}
\def\f{\frac}
\def\t{\hbox}

\def\q{\quad}
\def\ep{\varepsilon}
\def\s{\sigma}
\def\E{\tilde{E}}
\def\F{\mathcal{F}}

\begin{document}

\centerline{\bf A layer stripping method for numerical solution}

\centerline{\bf of the GPR problem in a layered medium}
\vspace{1cm}

\centerline{Ruben Airapetyan\footnote{Email: rhayrape@kettering.edu}}

\centerline{Kettering University, Flint 48504, USA}

\begin{abstract}
A numerical method for processing the data of ground penetrating radars for a piece-wise continuous layered medium is proposed. The method combines the layer stripping technique with numerical continuation of data into the complex frequency's domain. The accuracy of the method is analyzed. Error estimates are obtained and numerical testing is performed. They demonstrate numerical efficiency of the method under certain assumptions such as: electrical characteristics inside each layer change slowly, the thicknesses of the layers are at least of the order of the wavelength, conductivity of the medium is not high.
\end{abstract}

\noindent {\bf Keywords:}
Ground penetrating radar, piece-wise continuous layered medium, layer stripping, complex frequency domain, analytic continuation.

\noindent {\bf AMS Classification:}
65M32  35R30  86A22 34A55

\section{Introduction.}

In this paper a numerical method for solution of an inverse problem of the theory of ground penetrating radars is proposed. The ground penetrating radar system consists of a transmitter located over the surface and a receiver located on the surface. A pulsed electromagnetic wave is transmitted. The receiver measures the electric field. The goal of the GPR is to recover the electric  characteristics of the medium under the surface (\cite{dan}--\cite{persico}). 
Processing GPR data is an ill-posed problem that can be solved numerically only under some special assumptions about the medium. We refer to \cite{romkab} and citations there for detailed description of ill-posed and inverse problems for Maxwell's equations which aim to recover the electromagnetic material properties of a body from measurements on the boundary. There are many research articles where numerical solutions of the GPR problem are presented.   Most recent publications describe methods that are based on  minimizing the difference between the output calculated for the predicted model and the required output (\cite{knss,nag}). However, because of ill-posedness of the inversion problem this optimization problem requires regularization. For the full-waveform inversion method (\cite{bkv}--\cite{repp}) different regularization procedures are proposed (see \cite{llcw} and references there).  These methods employ numerical solvers for the corresponding direct problem and regularized optimization algorithms. Numerical results testify efficiency of these methods, however, they require good initial approximation, otherwise they may converge to a local minimum. Layer stripping methods represent an alternative approach to solution of GPR inverse problems. They do not require an initial approximation and computationally are less expensive. They can also be combined with the full-waveform inversion methods providing an initial approximation.

For a layered medium the values of the electric and magnetic components of the field on the top of every layer are sums of incident waves and waves reflected from the interfaces of the layers. In layer stripping methods one starts with the first layer, finds its electric characteristics, and then evaluates fields at the top of the next layer (\cite{aa}--\cite{
sylv2}). Then one repeats the same procedure for next layers. However, in a real frequency domain the inputs of reflected waves make considerable contribution to the total field and cannot be neglected. To avoid this difficulty complex frequencies can be used (\cite{aa,sylv1,sylv2,pratt1,pratt2}). If the imaginary part of the complex frequency is large enough then the result of multiple reflections from the lower layers can be suppressed and the field on the top of the layer can be approximated by an expression that includes the electric characteristics of the adjacent layer only. Thus, complex frequencies simulate a phenomenon of wave propagation in a medium with a strong energy loss when signals from the lower layers are relatively week. 

It is usually assumed in GPR systems that the electric characteristics of the background medium and, therefore, the propagation speed are a priori known. Then, the measurement of the time interval between transmitted and reflected signals allows one to evaluate the distance to the scatter. The goal of this work is to analyze a possibility of recovering both the electric characteristics of the background medium and the distances to the scatter without a priori information about the medium. There are many important applications when the medium is assumed to be planar-layered (\cite{bll1,bll2}).  Hence, it is assumed in this paper that the medium under the surface is  planar-layered with electric characteristics depending on the depth and that  the receiver measures the horizontal components of the electric and magnetic fields. These measurements contain the reflections from all planar interfaces. The goal of this work is to determine the electric properties and hence the type of the medium in each layer along with the thicknesses of the layers. 

For a piece-wise constant medium a new method based on analytical continuation of GPR data in the complex frequency domain is proposed in and tested in \cite{aa}.  In the framework of this method, the electric fields in the complex frequency domain are found using analytic continuation of the data initially given for real frequencies. Then, the impact of the first layer is separated and its electric characteristics and the wave propagation speed are determined. Once these characteristics are found the thickness of the first layer is approximated. Then, the direct problem for the first layer using the data on its top and already found characteristics of that layer is solved. After electric field and its normal derivative at the bottom of the first layer are found the same procedure is repeated for the next layers. Thus, for any given layer one first solves the inverse and then the direct problem. 

However, the assumption that the medium is piece-wise constant is too restrictive for practical applications. To this end, in this paper, the accuracy of the method is analyzed under a more realistic assumption that the permittivity and conductivity of the medium can vary inside each layer.  Thus, unlike \cite{aa,bugar}, electric characteristics of each layers are assumed to be piece-wise smooth functions that are approximated by piece-wise constant functions such that the change of the electric characteristics inside each layer is relatively small.  Since the problem is significantly ill-posed, the method allows approximation of the electrical characteristics with an accuracy that depends on such properties of the medium as the thickness of the layers, conductivity of the medium, and on variation of the permittivity inside the layers. Higher conductivity increases the energy decay and makes results less accurate. Also, very thin layers are impossible to detect, it is well known that to produce an observable reflection the thickness of the layer should be at least of the order of the wavelength. 

The efficiency of numerical methods for solution of GPR problems is usually justified by numerical experiments, but how it depends on the medium itself is not analyzed. However, due to the ill-posedness of the problem, the accuracy of a numerical solution strongly depends on the properties of the medium. Therefore, in this paper, for the proposed numerical method, the dependence of the accuracy of reconstruction of electric characteristics of a medium on the properties of the medium is analyzed both theoretically and numerically. The novelty of this paper is in the main result (Theorem \ref{thm1}) that contains estimates establishing dependence of the accuracy of the numerical solution on the properties of the medium  in the case of a piece-wise smooth layered medium. 

\section{Ground penetrating radars problem.}

Let $z$-axes be perpendicular to the surface $z=0$ and directed inside the medium. The assumption that the medium is planar-layered implies that the permittivity and magnetic permeability of medium are functions of $z$ only. Thus,
the permittivity is $\ep_0\ep(z)$ and magnetic permeability is $\mu_0\mu(z)$, where $\ep_0=8.854\cdot10^{-12}f/m$ and     
$\mu_0=1.257\cdot10^{-6}h/m$ are the permittivity and magnetic permeability in vacuum. For $z<0$ (above the surface) $\ep=1$, $\sigma=0$, and $\mu=1$.
Consider Maxwell's equations:
\begin{equation}\label{mxvec}
\nabla \times{\bf E}=-\mu\mu_0\f{\partial{\bf H}}{\partial t},\q
\nabla \times{\bf H}=\ep\ep_0\f{\partial{\bf E}}{\partial t}+\sigma{\bf E}+
{\bf j},
\end{equation}
where $\sigma$ is conductivity and ${\bf j}$
is an external source. 
\\  

\noindent{\bf GPR Problem. }~ For electric and magnetic fields ${\bf E}$ and
${\bf H}$ given on the surface $z=0$ find $\ep(z)$ and $\sigma(z)$ for
$0<z<L$,  where $L>0$ is the depth on which radar's radiation penetrates. \\ 

\noindent {\bf Remark.} Measurements of the magnetic field allow one to calculate the derivative of the electric field in the direction normal to the surface. Hence, in what follows, we assume that instead of the magnetic field that derivative is given on the surface.

Assume that ${\bf j}=f(t)\delta(x)\delta(z-z_0){\bf e}_y$, where ${\bf
  e}_y $ is a unit vector in direction of $y$-axes. Then the electric
field ${\bf E}$ does not depend on $y$ and is $y$-polarized. Denoting
$y$-component of {\bf E} by $E$  and replacing $\mu_0\ep_0$ by $c^{-2}$, where $c$ is the speed of light in vacuum, one gets the following equation: 
\begin{equation}\label{mxsc}
\f{\mu\ep(z)}{c^2}\f{\partial^2 E}{\partial t^2}-\f{\partial^2 E}{\partial
  x^2}-\f{\partial^2 E}{\partial z^2}+\mu_0\mu\sigma(z)\f{\partial
  E}{\partial t}+\mu_0\mu \f{\partial j}{\partial t}=0, 
\end{equation}
where 
$$
j=f(t)\d(x)\d(z-z_0),\q z_0<0.
$$
Denote by $T$ the transmission time, then $\t{supp}(f)=[0,T]$ and $f(0)=0$.

Applying Fourier transform
$$
\E(z) = \E(z,k_x,\om) =\f{1}{2\pi}\int\int e^{-i(\om t+k_x x)}E(z,x,t)dtdx
$$
one obtains:
\begin{equation}\label{eq1}
\f{d^2\E}{dz^2}+k_z^2(z)\E=h(\om)\d(z-z_0),
\end{equation}
where
\begin{equation}\label{fkz0}
k_z^2(z)=\f{\mu\ep(z)\om^2}{c^2}-i\mu_0\mu\om\s-k_x^2,\q h(\om)=i\mu_0\mu\om\tilde{f}(\om),\q 
\tilde{f}(\om)=\f{1}{\sqrt{2\pi}}\int e^{-i\om t}f(t)dt.
\end{equation}
Since $f=0$ for $t\le 0$ the electric field $u$ is zero for $t<0$. Hence, there exists analytic continuation of $\E$ and $h$ to the lower half of the complex plane $\om$. Let
\begin{equation}\label{2.6}%
\om_1:=\Re\om>0, \q \om_2:=\Im\om\le 0, \q  k_x= 0. 
\end{equation}
 Let
\begin{equation}\label{fkz}%
k_z(z)=-\sqrt{\mu\ep(z)c^{-2}\om^2-i\mu_0\mu\s\om}
\end{equation}
be the negative root if $k_z$ is real, and the root with positive
imaginary part if $k_z$ is complex. For $z<0$, $\ep(z)=1$ and $\sigma(z)=0 $, therefore
$$
k_z(z)=k^0=-\f{\om}{c},\q \hbox{for}\q z<0. 
$$
It is assumed that $\E$ satisfies the radiation conditions:
\begin{equation}\label{bc}
\lim_{z\to - \infty}\f{d\E(z)}{dz} + ik_z(z)\E(z)=0,\q \lim_{z\to \infty}\f{d\E(z)}{dz} - ik_z(z)\E(z)=0.
\end{equation}
Then it follows from (\ref{bc}) that 
\begin{eqnarray}\label{asym}
&&\E=C_1e^{-izk^0} \t{ for }-\infty<z<z_0\t{ and}\nonumber\\
&&\E=C_3e^{izk^0}+C_4e^{-izk^0} \t{ for } z_0<z<0. 
\end{eqnarray}
From equation (\ref{eq1}) one gets
$$
\E(z_0+0)=\E(z_0-0),\q \f{d\E}{dz}(z_0+0)=\f{d\E}{dz}(z_0-0)+h(\om), 
$$
and so
$$
C_3e^{iz_0k^0}+C_4e^{-iz_0k^0}=C_1e^{-iz_0k^0},
$$
$$
C_3e^{iz_0k^0}-C_4e^{-iz_0k^0}=-C_1e^{-iz_0k^0}-i h(\om)/k^0.
$$
Thus
$$
C_3=-\f{i h(\om)}{2k^0}e^{-iz_0k^0},\q C_4=C_1+\f{i h(\om)}{2k^0}e^{iz_0k^0}
$$
and for $z_0<z<0$ 
$$
\E=C_1e^{-izk^0}+\f{i h(\om)}{2k^0}e^{-i(z-z_0)k^0}-\f{i h(\om)}{2k^0}e^{i(z-z_0)k^0}.
$$
Thus,
$$
\f{d\E}{dz}(0)+ik^0\E(0)=h(\om)e^{-iz_0k^0}.
$$
Assume that $\ep$ and $\sigma$ are constant for $z>L$. Then one gets
$$
\E=C_2e^{izk_z(L)} \t{ for }L<z<\infty
$$
and 
$$
\f{d\E}{dz}(L)-ik_z(L)\E(L)=0.
$$
Thus $\E(z,\om)$ is the solution to the following problem:
\begin{equation}\label{stlv}\Biggl\{\ba{lcc}
\f{d^2\E}{dz^2}+k_z^2(z)\E=0,\\
(\f{d\E}{dz}+ik^0\E)_{z=0}= h(\om)e^{-ik^0 z_0},\q 
(\f{d\E}{dz}-ik_z(L)\E)_{z=L}=0.
\ea\Biggl.
\end{equation}

\section{GPR problem for a piece-wise smooth layered medium.}

The following lemma has been proved in \cite{aa}.

\begin{lemma}\label{lemma31} If
 \begin{equation}\label{12.6}
\om_1:=\Re\om>0, \q \om_2:=\Im\om< 0,
\end{equation}
then
$$
\E(z,\om)\neq 0,\q \Im(\E_z(z,\om)/\E(z,\om))\le 0\q \hbox{for}\q z\in [0,L].
$$
\label{l1}\end{lemma}
Since $\E(z,\om)\neq 0$ one can introduce a function
\begin{equation}\label{vz}
q(z,\om):=\f{\E_z(z,\om)}{\E(z,\om)}.
\end{equation}
From (\ref{stlv}) one gets the following problem for $q(z):=q(z,\om)$:
\begin{equation}\label{imeq}
q'(z)+q^2(z)+k_z^2(z)=0,\q q(L)=ik_z(L).
\end{equation}

\noindent{\bf Condition A.} There exists a positive constant $C$ such that
\be\label{epssig}
\f{\s(z)}{\ep(z)}\le\f{2C\om_1}{\mu_0c^2}\q\hbox{for}\q z\in [0,L],
\ee
\be\label{condA}
-(1+\phi(C))\om_1\le\om_2\le(1-\sqrt{2})\om_1,
\ee
where $\phi(C)=\sqrt{2+C^2}-C$.

\begin{lemma} Let Condition {\bf A} be satisfied, then 
\be\label{l11}
\arg(k_z^2)\in \left[-\f{3\pi}{4},-\f{\pi}{4}\right],\q \arg(k_z)\in \left[\f{5\pi}{8},\f{7\pi}{8}\right],
\ee
\be\label{l12}
\Re\{k_z\}<0,\q \Im\{k_z\}\ge\sqrt{\f{-2\mu\ep c^{-2}\om_1\om_2+\mu\mu_0\sigma\om_1}{2(\sqrt{2}+1)}}.
\ee
\label{argkz}\end{lemma}

Let $\ep(z)$ and $\s(z)$ be smooth functions on semi-closed intervals $I_j=[z_{j-1}, z_j)$ and on $I_{N+1}=[z_N,\infty)$ with finite jumps at $z_{j}$, $j=1\dots N$, where $0=z_0<z_1<\dots<z_N$, and $\ep(z)=\ep_j(z)$, $\s(z)=\s_j(z)$ on $I_j$, $\ep(z)=\ep_{N+1}=const$, $\s(z)=0$ on $I_{N+1}$. Then, the continuous function $q(z)$ satisfies the following equations: 
$$
q'(z)+q^2(z)+k_j^2(z)=0\q\hbox{ on }I_j,\q j=1,\dots,N,
$$
$$
q(z)=ik_{N+1}=-\f{\om}{c}\sqrt{\mu\ep_{N+1}}\q\hbox{ on }I_{N+1},
$$
where
$$
k_j(z)=-\sqrt{\mu\ep_j(z)c^{-2}\om^2-i\mu_0\mu\s_j(z)\om}.
$$
Denote $w_j=\f{u_j-1}{u_i+1}$, where $u_j(z)=-iq(z)/k_j(z)$, $z\in I_j$, $j=1,2,\dots,N+1$, then in $I_j$
\be\label{pweq1}
ik_ju'_j+ik'_ju_j-k_j^2u_j^2+k_j^2=0
\ee
and
\be\label{pweq2}
w'_j+2ik_jw_j-\f{k'_j}{2k_j}w_j^2+\f{k'_j}{2k_j}=0.
\ee
Since $q(z)$ is a continuous function,
\be\label{contcond}
\lim\limits_{z\to z_{j}-0}k_{j}(z)u_{j}(z)=k_{j+1}(z_{j})u_{j+1}(z_{j}), \q j=1,2,\dots,N.
\ee
Let
\be\label{c0}
c_0:=\tan\left(\f{3\pi}{8}\right)+\sec\left(\f{3\pi}{8}\right)=1+\sqrt{2}+\sqrt{4+2\sqrt{2}}\approx 5.
\ee

\noindent{\bf Condition B.}  For $\d\in (0,\sqrt{2}-1]$, $\om=\om_1+i\om_2$, 
and $j=1,2,\dots,N$, there exist constants $C_j$ and $\l_j$ such that for $z\in I_j$
\be\label{l421}
\f{|\ep'_j(z)|}{\ep_j(z)}\le 4\lambda_j\f{|\om_1\om_2|}{|\om|^2},\q
\f{|\s'_j(z)|}{\s_j(z)}\le 2\lambda_j\frac{\om_1}{|\om|},
\ee
\be\label{l422}
4 c^{-1}\sqrt{\f{\mu\ep_j(z)\om_1|\om_2|}{\sqrt{2}+1}}-c_0\l_j\ge C_j>0,
\ee
$$
e^{-C_j(z_j-z_{j-1})}+\f{\l_j}{C_j}\left(1+\f{\l_j}{C_j}\right)\le \d^2.
$$

\begin{theorem}  Let Conditions {\bf A} and {\bf B} be satisfied,
then for $j=0,1,\dots,N$,
\be\label{mest}
|w_{j+1}(z_{j})|\le \delta,\q \max\limits_{z\in I_{j+1}}|w_{j+1}(z)|\le 1+\f{\l_{j+1}}{C_{j+1}},\q \lim\limits_{z\to z_j-0}|w_j(z)|\le 1,
\ee
\be\label{t1est}
k_j(z_{j-1})=-iq(z_{j-1})\left(1-\kappa_j\right),\q\hbox{where}\q|\kappa_j|\le\f{2\d}{1-\d}.
\ee
\label{thm1}\end{theorem}

\section{Proofs of Lemma \ref{argkz} and Theorem \ref{thm1}.}

{\bf Proof of Lemma \ref{argkz}.}
From (\ref{fkz})
\begin{equation}\label{z1}
\begin{array}{l}
k_z^2=\f{\mu\ep(z)}{c^2}(\om_1^2-\om_2^2)+\mu_0\mu\s\om_2+i(
\f{2\mu\ep}{c^2}\om_1\om_2-\mu_0\mu\s\om_1),
\end{array}
\end{equation}
and, therefore,
$$
\f{\Re(k_z^2)}{\Im(k_z^2)}=\f{\mu\ep c^{-2}(\om_1^2-\om_2^2)+\mu\mu_0\sigma\om_2}{2\mu\ep c^{-2}\om_1\om_2-\mu\mu_0\sigma\om_1}=\f{\l^2+\a\l-1}{2\l+\a},
$$
where $\l=-\f{\om_2}{\om_1}>0,\q \a=\f{\mu_0c^2\sigma}{\ep\om_1}\ge 0$.
Inequalities (\ref{epssig}) and (\ref{condA}) imply that $\sqrt{2}-1\le\l\le 1+\phi(C)$, $0\le\a\le 2C$ and, hence,
$
-1\le \f{\l^2+\a\l-1}{2\l+\a}\le 1.
$
Thus, 
\be\label{l102}
-1\le \f{\Re(k_z^2)}{\Im(k_z^2)}\le 1.
\ee
Since $\Im\{k_z^2\}<0$, this inequality implies that $\arg(k_z^2)\in \left[-\f{3\pi}{4},-\f{\pi}{4}\right]$. Hence, $\arg(k_z)\in \left[\f{5\pi}{8},\f{7\pi}{8}\right]$, since $\Im\{k_z\}>0$, and, therefore, $\Re\{k_z\}<0$.
Finally, using estimate (\ref{l102}) again, one gets
$$
\Im(k_z)=-\Im(k_z^2)(2(|k_z^2|+\Re(k_z^2)))^{-1/2}\ge\sqrt{\frac{-\Im(k_z^2)}{2(\sqrt{2}+1)}}.
$$

\noindent {\bf Proof of Theorem \ref{thm1}.} To prove estimates (\ref{mest}) the method of mathematical induction is used. First, we verify estimates (\ref{mest}) for $j=N$. Since $u_{N+1}=1$  and $w_{N+1}=0$ in $I_{N+1}$
 the first two inequalities hold.  To verify the third inequality we use condition (\ref{contcond}) and inequalities $\Re(k_z)<0$ and $\Im(k_z)\ge 0$ established in Lemma \ref{argkz}. Thus,
$$
\lim\limits_{z\to z_N-0}\Re\{u_N(z)\}=\lim\limits_{z\to z_N-0}\Re\left\{\f{k_{N+1}(z_N)}{k_N(z)}\right\}\ge 0,
$$
and, therefore,
$$
\lim\limits_{z\to z_N-0}|w_N(z)|^2 = \lim\limits_{z\to z_N-0}\f{|u_N(z)|^2-2\Re(u_N(z))+1}{|u_N(z)|^2+2\Re(u_N(z))+1}\le 1.
$$
Now we assume estimates (\ref{mest}) hold for $j=m+1$ and prove them for $j=m$.
Lemma \ref{lemma31} implies $\Re(k_ju)=\Im(q)\le 0$. Therefore,
$$
\Re\left(k_j\f{1+w_j}{1-w_j}\right)\le 0
$$
and, hence,
$$
\Re(k_j)(1-|w_j|^2)-2\Im(k_j)\Im(w_j)\le 0.
$$
Since $\Re(k_j)<0$, this inequality implies
$$
|w_j|^2\le 1+2\left|\f{\Im(k_j)}{\Re(k_j)}\right|\cdot |w_j|.
$$
After solving this inequality with respect to $|w_j|$, one gets
\be\label{estw1}
|w_j|\le \left|\f{\Im(k_j)}{\Re(k_j)}\right|+\sqrt{\left(\f{\Im(k_j)}{\Re(k_j)}\right)^2+1}.
\ee
According to  Lemma {\ref{argkz}}, $\arg(k_z)\in \left[\f{5\pi}{8},\f{7\pi}{8}\right]$. Therefore,
$$
\left|\f{\Im(k_j)}{\Re(k_j)}\right|\le \tan\left(\f{3\pi}{8}  \right).
$$
Denote,
$$
\b_j:=\max\limits_{z_{j-1}<z<z_j}|w_j(z)|,
$$
then (\ref{estw1}) and (\ref{c0}) imply the initial estimate for $\b_j$:
\be\label{betac0}
\b_j\le c_0.
\ee
  In $I_m$ equation (\ref{pweq2}) implies
$$
\frac{d}{dz}|w_m|^2-4\Im(k_m)|w_m|^2-\Re\left(\f{k'_m}{k_m}w_m\right)|w_m|^2+\Re\left(\f{k'_m}{k_m}\overline{w_m}\right)=0.
$$
Applying to this inequality the estimate  $|w_m(z)|\le \b_m$ yields
\be\label{IntIneq}
\frac{d}{dz}|w_m|^2-(4\Im(k_m)-\beta_m|k'_m/k_m|)|w_m|^2\ge -\beta_m|k'_m/k_m|.
\ee
Denote
$$
\g_m(z):=4\Im(k_m)-\beta_m|k'_m/k_m|.
$$
Since $\om_1>0$, $\om_2\le0$, 
$$
2|k_z'k_z|=\left|\f{d}{dz}k_z^2\right|=\left|\frac{\mu\ep'\om^2}{c^2}-i\mu_0\mu\om\s'\right|\le \frac{\mu|\om|^2}{c^2}|\ep'|+\mu_0\mu|\om\s'|
$$
$$
\le 2\l_m\left(-\frac{2\mu\ep}{c^2}\om_1\om_2+\mu_0\mu\om_1\s\right)
=2\l_m(-\Im(k_z^2))\le 2\l_m|k_z|^2,
$$
and, hence, $\left|\f{k'_z(z)}{k_z(z)}\right|\le \l_m$ in $I_m$.
Therefore, inequalities (\ref{l12}), (\ref{l422}), and (\ref{betac0}) imply that
\be\label{estGamma}
\g_m(z)=4\Im(k_m)-\beta_m|k'_m/k_m| \ge 4 c^{-1}\sqrt{\f{\mu\ep_m(z)\om_1|\om_2|}{\sqrt{2}+1}} -c_0\l_m\ge C_m> 0.
\ee
The multiplication of (\ref{IntIneq}) by $\exp\left\{\int\limits_z^{z_m}\g_m(s)ds\right\}$ yields
$$
\frac{d}{dz}\left(|w_m(z)|^2\exp\left\{\int\limits_z^{z_m}\g_m(z)ds\right\}\right)
\ge -\beta_m|k'_m/k_m|\exp\left\{\int\limits_z^{z_m}\g_m(s)ds\right\},
$$
and, after integration with respect to $z$ one gets the following estimate
$$
|w_m(z)|^2\le \lim\limits_{\zeta\to z_m-0}|w_m(\zeta)|^2\exp\left\{-\int\limits_z^{z_m}\g_m(s)ds\right\}
+\beta_m\int\limits_z^{z_m}\left|\f{k'_m(\tau)}{k_m(\tau)}\right|\exp\left\{-\int\limits_z^\tau\g_m(s)ds\right\}d\tau.
$$
Now we estimate the last term in the inequality above:
$$
\beta_m\int\limits_z^{z_m}\left|\f{k'_m(\tau)}{k_m(\tau)\g_m(\tau)}\right|\g_m(\tau)\exp\left\{-\int\limits_z^\tau\g_m(s)ds\right\}d\tau
$$
$$
\le\b_m \max\limits_{\zeta\in I_m}\left\{\f{|k'_m(\zeta)/k_m(\zeta)|}{\g_m(\zeta)}\right\}\left(1-\exp\left\{-\int\limits_z^{z_m}(\g_m(s)ds\right\}\right)\le \l_m\b_m\max\limits_{\zeta\in I_m}\g_m^{-1}(\zeta).
$$
Using the induction assumption that $ \lim\limits_{\zeta\to z_m}|w_m(\zeta)|\le 1$ and estimate (\ref{estGamma}) we get that
\be\label{west}
|w_m(z)|^2\le\exp\left\{-\int\limits_z^{z_m}\g_m(s)ds\right\}+\f{\b_m\l_m}{C_m}\le 1+\f{\l_m\b_m}{C_m}.
\ee
Hence,
$$
\b_m^2\le 1+\f{\l_m\b_m}{C_m},
$$
and, therefore,
$$
|w_m(z)|\le \b_m\le\f{\l_m}{2C_m}+\sqrt{\f{\l_m^2}{4C_m^2}+1}\le 1+\f{\l_m}{C_m}.
$$
Therefore, estimate (\ref{west}) implies
$$
|w(z_{m-1})|^2\le \exp\left\{-\int\limits_{z_{m-1}}^{z_m}\g_m(s)ds\right\}+\f{\l_m}{C_m}\left(1+\f{\l_m}{C_m}\right) 
$$
$$
\le e^{-C_m(z_m-z_{m-1})}+\f{\l_m}{C_m}\left(1+\f{\l_m}{C_m}\right)\le\d^2.
$$
Thus, for $j=m$ the first two inequalities in (\ref{mest}) hold. In order to prove the third inequality we use condition  (\ref{contcond}):
$$
\lim\limits_{z\to z_{m-1}-0}k_{m-1}(z)u_{m-1}(z)=k_{m}(z_{m-1})u_{m}(z_{m-1}).
$$
Denote
$$
k_{m-1}^*:=\lim\limits_{z\to z_{m-1}-0}k_{m-1}(z),
$$
then
$$
\lim\limits_{z\to z_{m-1}-0}u_{m-1}(z)=\f{k_{m}(z_{m-1})u_{m}(z_{m-1})}{k^*_{m-1}}.
$$
After taking the real parts in the formula above we get
$$
\lim\limits_{z\to z_{m-1}-0}\Re(u_{m-1}(z))
$$
$$
=|k^*_{m-1}|^{-2}[\Re(k_{m}(z_{m-1})\overline{k^*_{m-1}})\Re(u_{m}(z_{m-1})-\Im(k_{m}(z_{m-1})\overline{k^*_{m-1}})\Im(u_{m}(z_{m-1})],
$$
where, according to Lemma \ref{argkz}, $\arg(k_{m}(z_{m-1})\overline{k^*_{m-1}})\in [-\pi/4,\pi/4]$ and, hence,
$$
\Re(k_{m}(z_{m-1})\overline{k^*_{m-1}})\ge 0\hbox{ and } \Re(k_{m}(z_{m-1})\overline{k^*_{m-1}})\ge \Im(k_{m}(z_{m-1})\overline{k^*_{m-1}}).
$$
Further, $|w_m(z_{m-1})|\le \d\le\sqrt{2}-1$ and
$$
u_m(z_{m-1})=\f{1+w_m(z_{m-1})}{1-w_m(z_{m-1})}=\f{1-|w_m(z_{m-1})|^2+2i\Im(w_m(z_{m-1}))}{|1-w_m(z_{m-1})|^2},
$$
imply
$$
\Re(u_m(z_{m-1}))\ge 0\hbox{ and }\Re(u_m(z_{m-1}))\ge\Im(u_m(z_{m-1})).
$$
Therefore,
$$
\lim\limits_{z\to z_{m-1}-0}\Re(u_{m-1}(z))\ge 0
$$
and, hence,
$$
\lim\limits_{z\to z_{m-1}-0}|w_{m-1}(z)|^2 
= \lim\limits_{z\to z_{m-1}-0}\f{|u_{m-1}(z)|^2-2\Re(u_{m-1}(z))+1}{|u_{m-1}(z)|^2+2\Re(u_{m-1}(z))+1}\le 1.
$$
Thus, estimates (\ref{mest}) are proven by the mathematical induction. Finally,
$$
k_z(z_{j-1})=\f{q(z_{j-1})}{iu_j(z_{j-1})}=-iq(z_{j-1})\f{1-w_j(z_{j-1})}{1+w_j(z_{j-1})}=-iq(z_{j-1})\left(1-\f{2w_j(z_{j-1})}{1+wj(z_{j-1})}\right)
$$
and
$$
|w_j(z_{j-1})|\le \d
$$
imply
$$
|\kappa_j|=\left|\f{2w(z_{j-1})}{1+w(z_{j-1})}\right|\le\f{2|w(z_{j-1})|}{1-|w(z_{j-1})|}\le\f{2\d}{1-\d}.
$$


\section{Inversion method.}

From (\ref{fkz0})
\be\label{forep}
\ep_j(z_{j-1})=\f{c^2}{\mu|\om|^2}\left(\Re\{k^2_j(z_{j-1}\}+\f{\om_2}{\om_1}\Im\{k^2_j(z_{j-1})\}\right),
\ee
\be\label{fors}
\s_j(z_{j-1})=\f{\om_2}{\mu_0\mu|\om|^2}\left(2\Re\{k^2_j(z_{j-1}\}+\left(\f{\om_1}{\om_2}-\f{\om_2}{\om_1}\right)\Im\{k^2_j(z_{j-1})\}\right).
\ee
Theorem \ref{thm1} yields that if conditions {\bf A} and {\bf B}  are satisfied, then an appropriate choice of $\omega_1$
and $\omega_2$  makes the error term $\kappa$ small and, hence, one can find an approximate value of $k_z$ at the top of the layer:
\be\label{kzq}
k^2_j(z_{j-1},\om)\approx -q^2(z_{j-1},\om) =-\left(\f{\E_z(z_{j-1},\om)}{\E(z_{j-1},\om)}\right)^2.
\ee
 Therefore, one can
determine $\ep$, $\s$, and the  and wave propagation speed $c/\sqrt{\ep}$ for that
layer. Note, that there is an assumption that the $|\ep'_j(z)|$ is small and, therefore, the variation of $\ep(z)$ inside the layer is small as well.
 Measuring the time interval between impulses corresponding to
incident and reflected waves one find the thickness of the
layer. Now, when $\ep$, $\s$, and the thickness of the given layer are known, one can choose a real frequency $\om=\om_1$, 
solve the direct problem, and evaluate $\E$ and $d\E/dz$ at the bottom of the layer that is the top of the next layer. 
Then, the same procedure can be repeated for the next layer.  

An analytic continuation of
functions $\E(z_{j-1},\om)$ and $\E_z(z_{j-1},\om)$   to the complex frequency domain 
$\om=\om_1+i\om_2$ is the most essential step for numerical implementation of this scheme.
Such analytic continuation is an ill-posed procedure and should be done very carefully. In our numerical experiments fast Fourier-Laplace transform was used.  For the next layer (say for the layer number $j$) $E=0$ and $dE/dz=0$ for $t<t_j$,
where $t_j$ is the time of the first arrival of the wave to the top of
the layer. In order to improve the accuracy of numerical analytic continuation
a rescaling $t\to t-t_j$ is used for every layer.  A regularized stable analytic continuation from the real axis is proposed in \cite{fffq}.

The choice of the imaginary part $\om_2$ of the frequency $\om$ is crucial. In order to satisfy inequality (\ref{condA}) in Condition A , $\om_2$ can be chosen in the interval $[-\om_1,(1-\sqrt{2})\om_1]$. The smaller values of $\om_2$ (the values with larger absolute values) such as $\om_2=-\om_1$ decrease the error term $\kappa$ and, hence, increase the accuracy of approximation in (\ref{kzq}), but, at the same time, they make the numerical analytic continuation less stable and increase the corresponding error. In our numerical experiments for an analytic continuation the fast Fourier transform was used and $\om_2=-0.9\om_1$ was chosen. However, a more efficient analytic continuation procedure would allow to achieve higher accuracy by choosing $\om_2=-\om_1$.

\section{ Numerical algorithm.} Below we assume that $\mu$ is known. Usually one can assume that for a non-magnetic medium $\mu=1$. Let $N$ be the number of layers.  Proposed numerical
algorithm includes the following steps repeated for every $j$ from $1$
to $N$. 
\begin{itemize}

\item I step: using the data on the top of the layer $[z_{j-1},z_{j}]$
  (for $j=1$ these data are obtained from the measurements on the
  surface, and for $j> 1$  from calculations done for
  the previous layer) use Fast Fourier Transform (FFT) to compute $\E(z_{j-1},\omega_1)$:
\be\label{na1}
\E(z_{j-1},\omega_1)=\F\{E(z_{j-1},x,t)\}:=\f{1}{2\pi}\iint e^{-i\omega_1t}E(z_{j-1},x,t)dtdx; 
\ee
\item  II step:  choose the real frequency $\omega_1$ such that $|\E(z_{j-1},\om_1)|$ is maximal;
\item  III step:  choose an imaginary part of the frequency $\om_2=-0.9\om_1$;
\item IV step: use FFT to find analytic continuations of $\E(z_{j-1},\om)$ and $\E_z(z_{j-1},\om)$:
$$
\E(z_{j-1},\omega_1+i\omega_2)=\F\{e^{-\om_2t}E(z_{j-1},x,t\},
$$
$$
\E_z(z_{j-1},\omega_1+i\omega_2)=\F\{e^{-\om_2t}E_z(z_{j-1},x,t\}; 
$$
\item V step: compute $k_j(z_{j-1)}$:
\begin{equation}\label{na3}
k_j(z_{j-1)}=-iq_{j-1}=-i\f{\E_z(z_{j-1},\omega_1+i\omega_2)}{\E(z_{j-1},\omega_1+i\omega_2)}; 
\end{equation}
\item
VI step: use formulas (\ref{forep}) and (\ref{fors}) to find $\ep(j)=\ep_j(z_{j-1})$, $\s(j)=\s_j(z_{j-1})$, and the wave speed $\f{1}{\sqrt{\mu\ep(j)}}$;
\item VII step: to compute the thickness of the layer measure the
time interval $\Delta t_j$ between the two corresponding impulses and
use the wave speed from the previous step;
\item VIII step: choose $\ep(j)$ and $\s(j)$ from step VI, $\om=\om_1$ ($\om_2=0$), and use the differential equation (\ref{stlv}) and known  $\E(z_{j-1},\om_1)$ and $\E_z(z_{j-1},\om_1)$ to compute $\E(z_j,\om_1)$ and $\E_z(z_j,\om_1)$; 
\item IX step: use IFFT (the Fast Inverse Fourier Transform) to find $E(z_j,t)$ and
$E_z(z_j,t)$;
\item X step: rescale time $t\to t-\Delta t$, $E(z_j,t-\Delta t)\to E(z_j,t)$ , $E_z(z_j,t-\Delta t)\to E_z(z_j,t)$.
Rescaling of time is needed in order to get a new time scale, where $t=0$
corresponds to the first arrival of the impulse to the top of the
layer.
\end{itemize}

\section{Numerical results.}

For numerical testing of the method  the Ricker wavelet with the central frequency 200 MHz was used for the original pulse generated by a transmitter. The simulated medium with the total depth $L$ about 84 m was chosen. In the following figures the results of simulation and reconstruction of relative permittivity are presented in two cases. In the first case the conductivity $\s$ is about $10^{-4}$ S/m, and in the second case the conductivity $\s$ is about $10^{-8}$ S/m.  By solving the direct problem synthetic data were generated. In the interval $[0,L]$ 200 points were chosen with the step $\Delta z=0.42$. In the time interval $[0,T]$, $T=0.005$ s., $2^{24}$ points were chosen with the step  $3\cdot10^{-10}$ s. The chosen interval of the frequencies was $[0,\Omega]$, $\Omega=20$ GHz and the step $\Delta\omega=1.22$ kHz.
Figures \ref{fig1}, \ref{fig2}, \ref{fig3} show the original relative permittivity $\ep$ and the results of numerical reconstruction. As it is expected, the lower conductivity allows better reconstruction.
\vspace{1cm}

\begin{figure}[h]
\centering
\subfloat[Relative permittivity $\varepsilon$ simulated.]{\includegraphics[width=5.5cm,height=4cm]{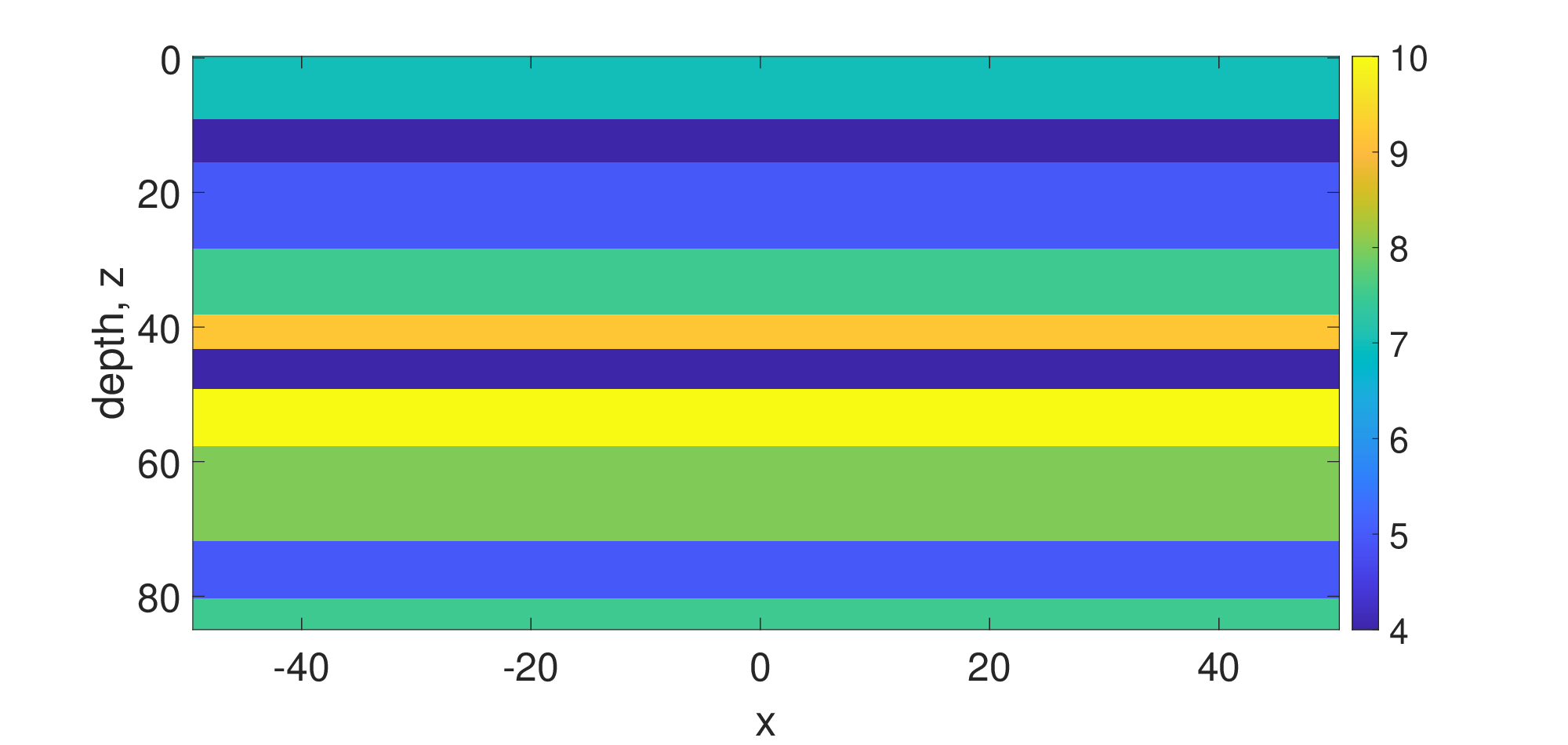}}\hspace{3pt}
\subfloat[Relative permittivity $\varepsilon$ reconstructed.]{\includegraphics[width=5.5cm,height=4cm]{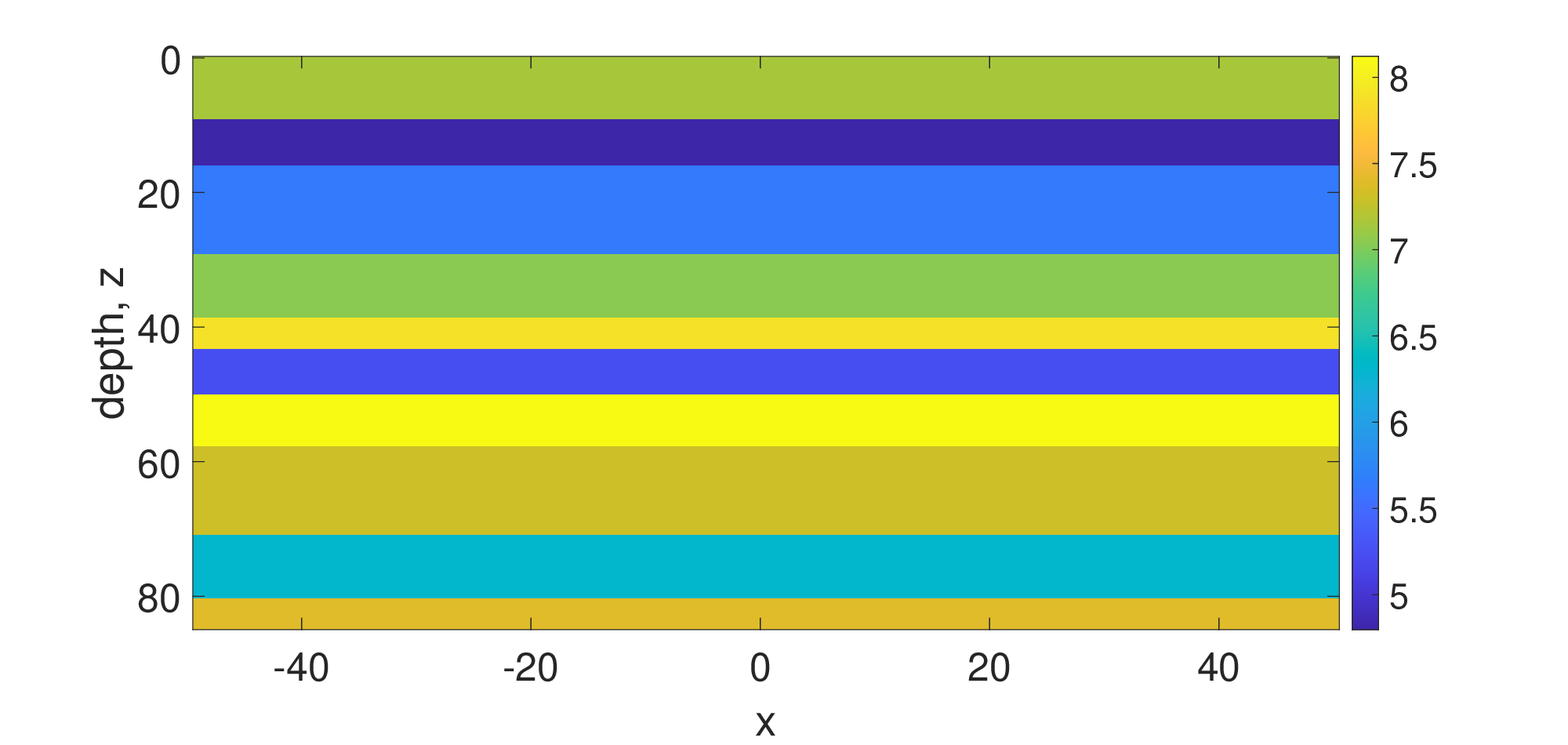}}
\caption{Relative permittivity simulated and reconstructed when the conductivity $\s$ is about $10^{-4}$ S/m. } \label{fig1}
\end{figure}

\begin{figure}[h]
\centering
\subfloat[Relative permittivity $\varepsilon$ simulated.]
{\includegraphics[width=5.5cm,height=4cm]{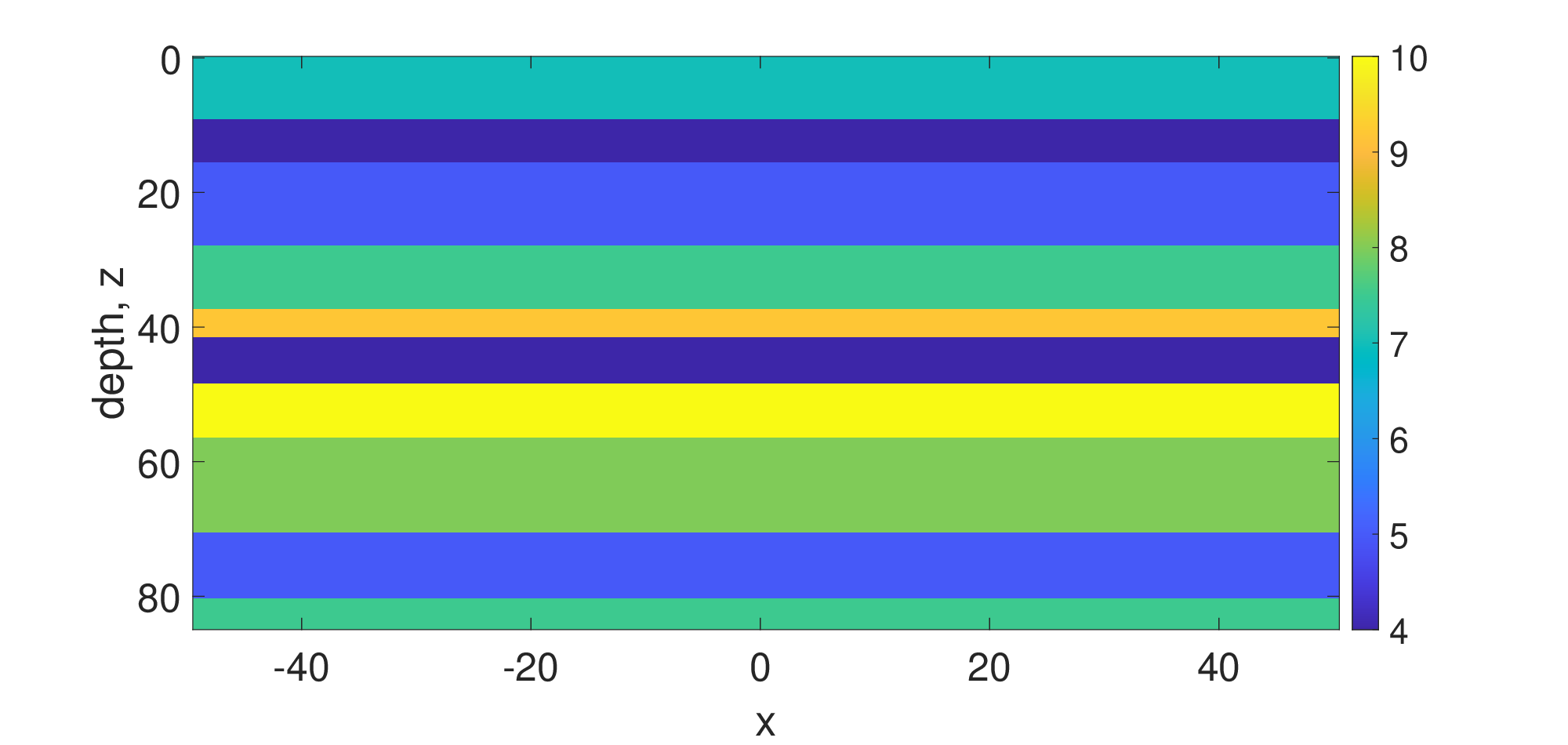}}\hspace{3pt}
\subfloat[Relative permittivity $\varepsilon$ reconstructed.]
{\includegraphics[width=5.5cm,height=4cm]{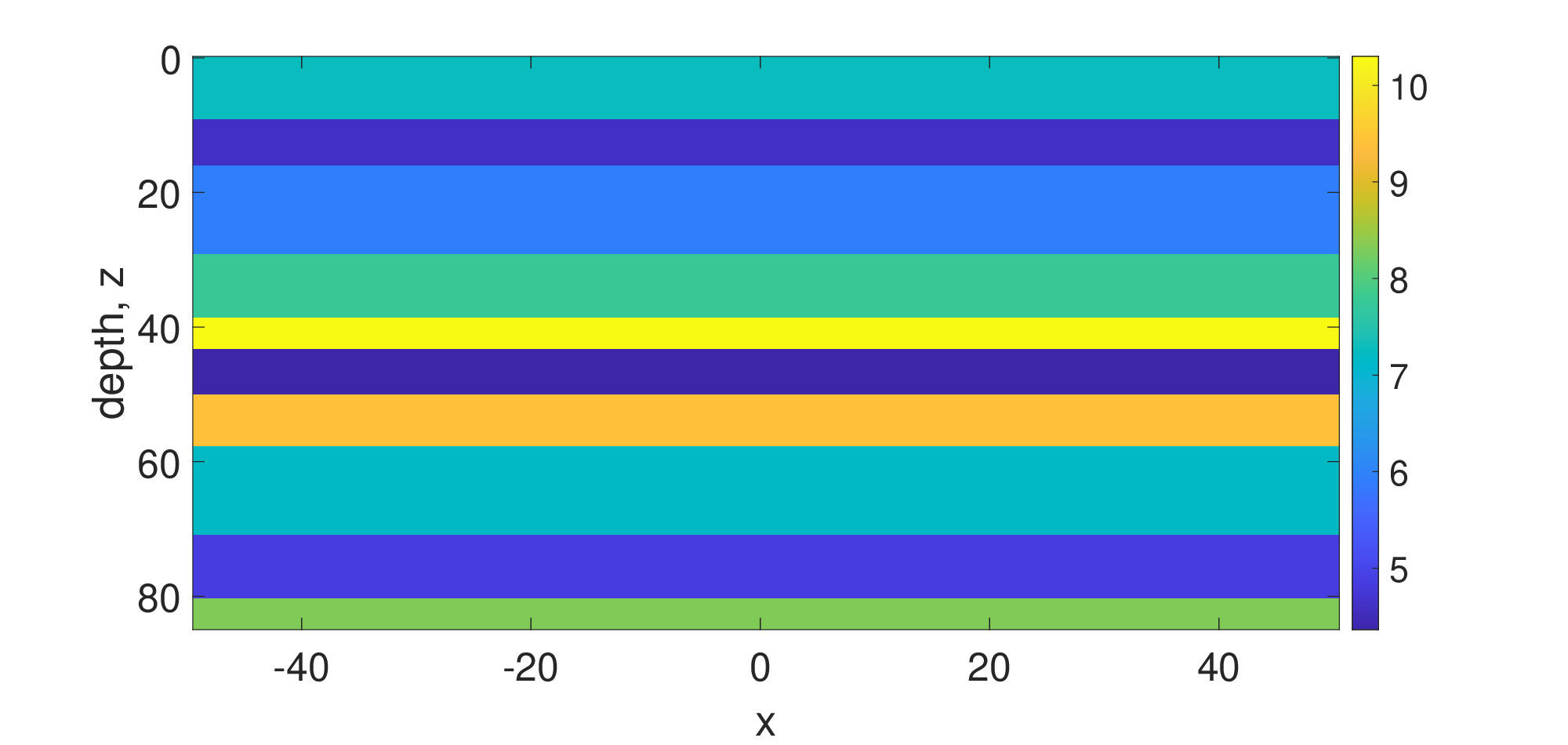}}
\caption{Relative permittivity simulated and reconstructed when the conductivity $\s$ is about $10^{-8}$ S/m. } \label{fig2}
\end{figure}

\begin{figure}[h]
\centering
\subfloat[The conductivity $\s$ is about $10^{-4}$ S/m. ]{\includegraphics[width=5.5cm,height=4cm]{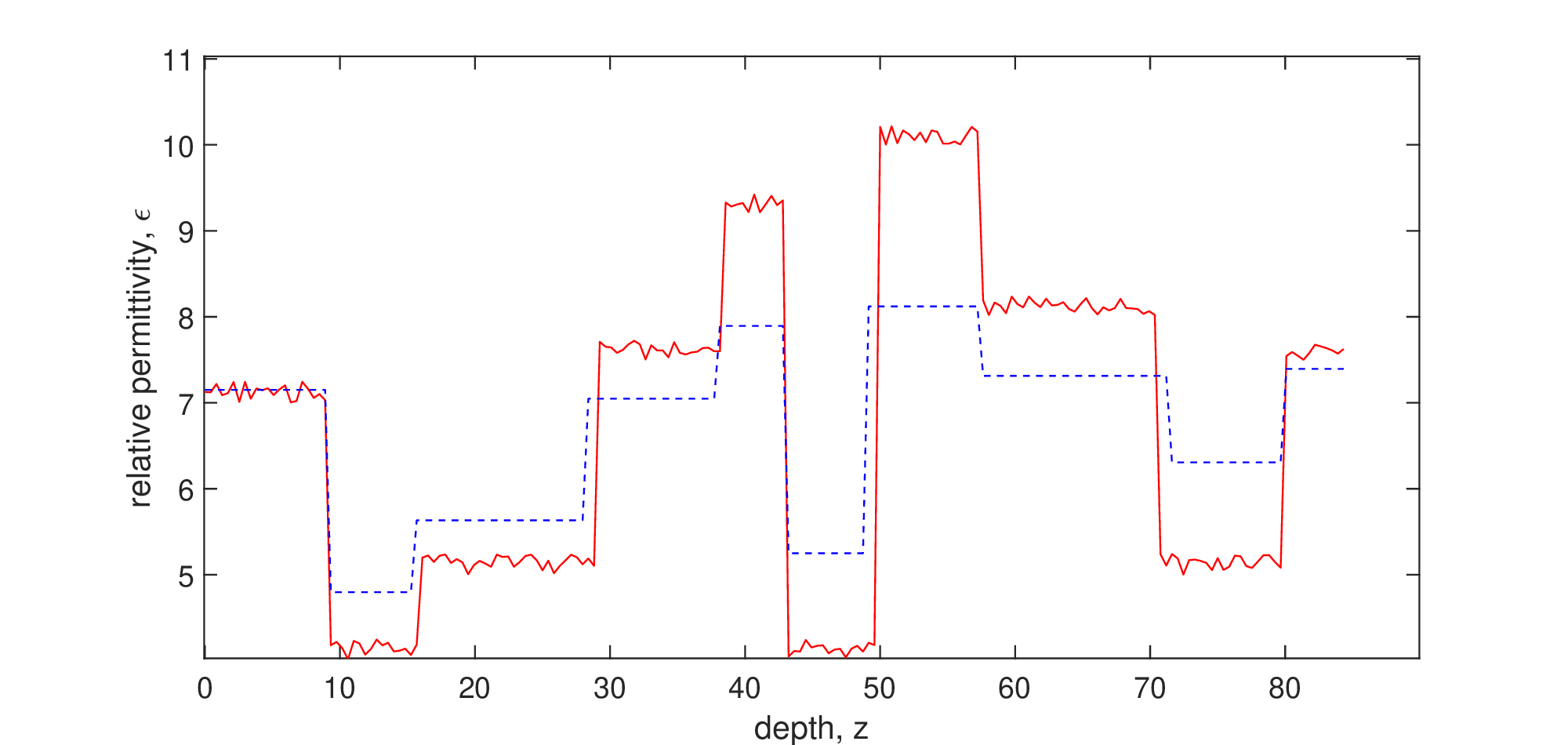}}\hspace{3pt}
\subfloat[The conductivity $\s$ is about $10^{-8}$ S/m. ]{\includegraphics[width=5.5cm,height=4cm]{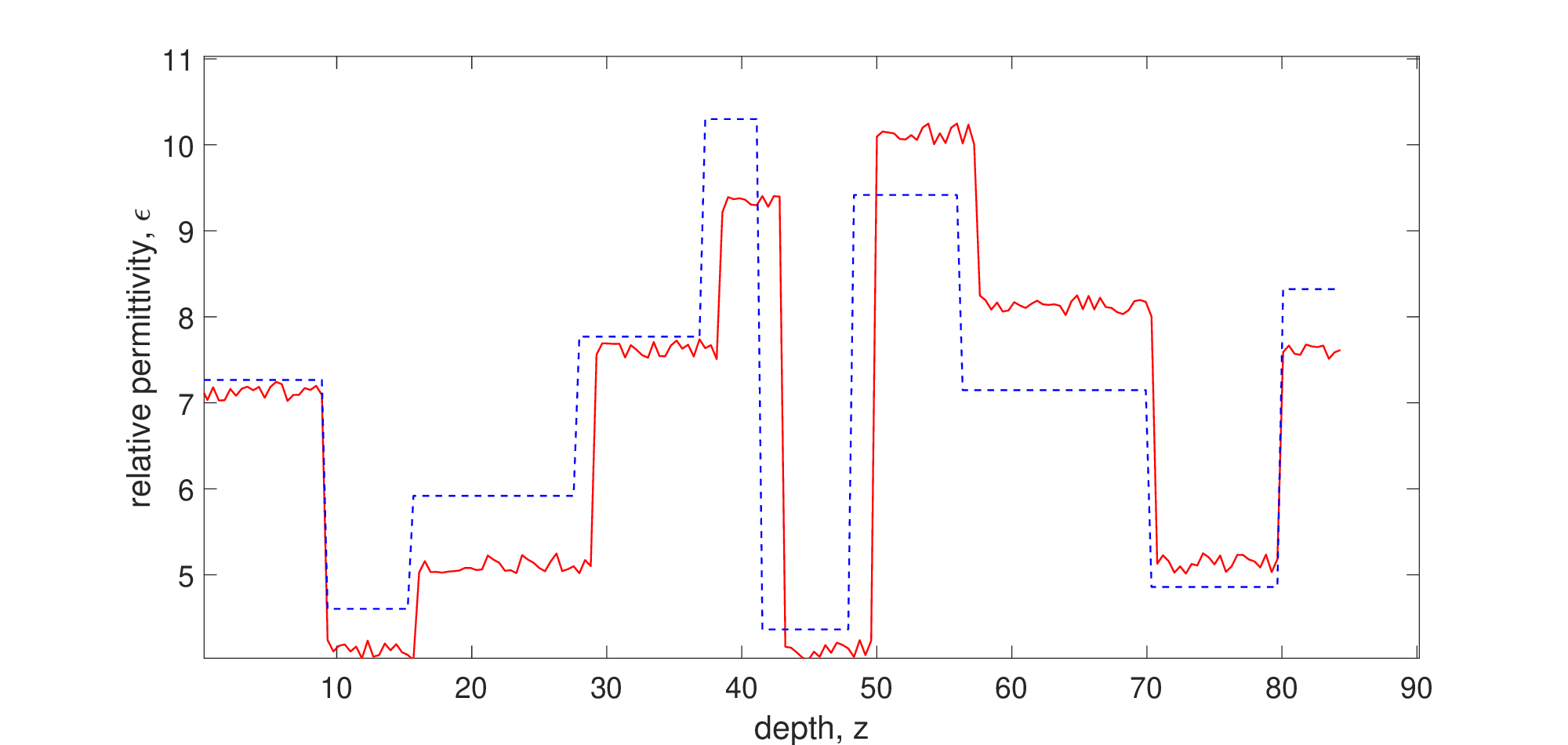}}
\caption{Relative permittivity simulated (red solid line) and reconstructed (blue dashed line). } \label{fig3}
\end{figure}

\section{Conclusion.}

In this article the numerical method for processing the data of ground penetrating radars, proposed in \cite{aa} for a piece-wise constant layered medium, is analyzed and numerically tested for a more realistic piece-wice smooth layered medium. Since electrical characteristics of the medium vary inside the layers, the accuracy of the reconstruction is lower than in the case of a piece-wise constant layered medium. The goal of this paper is to show that the method is still applicable in this case. As a result of a significant ill-posedness of the problem the accuracy of reconstruction of the medium strongly depends on electric characteristics of the medium itself.  The thickness of the layers, conductivity of the medium, and variation of the permittivity inside the layers strongly affect the accuracy of reconstruction. Higher conductivity increases the energy decay and makes results less accurate. Also, very thin layers are impossible to detect, it is well known that to produce an observable reflection the thickness of the layer should be at least of order of the wavelength. 
To justify the numerical effectiveness of the method estimates showing dependence of accuracy of the numerical solution on the properties of the medium are derived (Theorem \ref{thm1}). Numerical experiments also show efficiency of the method under certain assumptions, such as, electrical characteristics inside each layer change slowly, the thickness of layers is, at least, of the order of the wavelength, conductivity of the medium is not high. For the piece-wise continuous medium modelled above numerical errors of reconstruction of the permittivity of the first layer is about 5\% and is increasing for the next layers. However, as it can be seen on Figures 1 and 2 the reconstructed permittivity and the thickness of layers present a good image of the modelled layered medium.



\end{document}